\documentclass{amsart}

\usepackage{bbold}
\usepackage{amsthm}
\usepackage{amssymb}

\usepackage[style=numeric-comp]{biblatex}
\newcommand\kk{\mathbb{k}}
\newcommand\RR{\mathbb{R}}
\newcommand\NN{\mathbb{N}}
\newcommand\dfn\textbf
\newcommand\ct\mathbf
\newcommand\Vect{\mathop{\mathrm{Vect}}}
\newcommand\Id{\mathop{\mathrm{Id}}}
\newcommand\fromto{\rightleftarrows} 

\newtheorem{proposition}{Proposition}
\newtheorem{corollary}[proposition]{Corollary}
\newtheorem{theorem}[proposition]{Theorem}

\begin{document}

\author{Mikael Vejdemo-Johansson}
\title[Interleaved equivalence of categories of persistence modules]
{Interleaved equivalence of \\ categories of persistence modules}

\begin{abstract}
  We demonstrate that an equivalence of categories using $\varepsilon$-interleavings as a fundamental component exists between the model of persistence modules as graded modules over a polynomial ring and the model of persistence modules as modules over the total order of the real numbers.
\end{abstract}

\maketitle

\section{Introduction}
\label{sec:introduction}

Introduced by \textcite{edelsbrunner_topological_2000}, persistent homology has seen significant success in making homology directly applicable to problems from data analysis and graphics. A good overview of the development of the field has been written by \textcite{carlsson_topology_2009}.

Fundamentally, everyone agree on what it is persistent homology is supposed to model: given a topological space $X$ with a function $X\to\RR$, we can consider the sequence of sublevel sets: $X_t = f^{-1}(-\infty, t]$. This sequence is a filtration of the original space, and assuming $f$ fulfills certain tameness conditions can be discretized to a finite filtration of $X$.

The field has developed in two distinct directions, characterized by the choice of two different families of categories to represent the underlying algebra of persistent homology. In one direction, persistence modules is considered to be representations of the total order of the real numbers, and attention is paid to analytic style questions -- in particular to questions of \emph{stability} (this viewpoint is dominant in \cite{crawley-boevey_decomposition_2012,chazal_proximity_2009,cohen-steiner_stability_2007,chazal_stability_2008,chazal_structure_2012}). In the other direction, the process is discretized very early on, and attention is paid to algebraically natural extensions and algorithms (this is the viewpoint used in \cite{chen_output-sensitive_2011,bendich_analyzing_2008,gamble_applied_2012,ghrist_barcodes:_2008,wang_branching_2011,chacholski_combinatorial_2012,zomorodian_computing_2005,heras_computing_2012,de_silva_coordinate-free_2006,de_silva_coverage_2007,de_silva_dualities_2011,tausz_javaplex:_2012,zomorodian_localized_2007,carlsson_local_2007,carlsson_persistence_2005,morozov_persistent_2011,cohen-steiner_persistent_2009,carlsson_theory_2009,de_silva_topological_2004,carlsson_topology_2009,weinberger_what_2011,carlsson_zigzag_2008,carlsson_zigzag_2009,milosavljevic_zigzag_2009}).

The two directions both provide important facts to the analysis sequence: the analytic style questions enable stability results that provide a style of inference for persistent homology as a data analysis approach. The discretized approach identifies persistence modules with either modules over a quiver of type $A_n$, or graded modules over the polynomial ring $\kk[t]$ in one variable. In either case, the algebraic characterization implies the decomposition of a persistence module into \emph{interval modules}, and enable the description of persistent homology by \emph{barcodes} or \emph{persistence diagrams}.

In this paper, we will prove the equivalence of the two approaches by describing a weakened form of equivalence of categories, and then demonstrate that such a weakened equivalence exists between the models in use in the two cultures.

\section{Background}
\label{sec:background}


We shall recall some fundamental notions that are used in this paper. A \dfn{category} is a collection of objects and morphisms, where a morphism is associated to a source and a a target object. Morphisms $f, g$ such that the source of $f$ is the target of $g$ compose to a new morphism $fg$ sharing source with $g$ and target with $f$. This composition is associative, and each object $x$ has an identity morphism $1_x$ such that $1_xf = f1_y = f$ for any $f:x\to y$.

An isomorphism is a morphism with an inverse: $f:x\to y$ is an isomorphism if there is some $g:y\to x$ such that $fg=1_y$ and $gf=1_x$.

A \dfn{functor} is a map between categories. It takes objects to objects and morphisms to morphisms respecting composition. One example is the \dfn{trivial functor} $1_{\ct C}$ that acts with the identity map on both objects and morphisms.

For functors $\mathcal F, \mathcal G: \ct C\to \ct D$, we define a \dfn{natural transformation} $\phi: \mathcal F\Rightarrow\mathcal G$ to be a collection of morphisms $\phi_c:\mathcal F(c)\to\mathcal G(c)$ such that they commute with lifted morphisms: $\phi_d\mathcal F(c\to d) = \mathcal G(c\to d)\phi_c$. If all the morphisms $\phi_c$ are isomorphisms, we call $\phi$ a \dfn{natural isomorphism}.

Two categories $\ct C, \ct D$ are \dfn{equivalent} whenever there are functors $\mathcal F:\ct C\fromto \ct D:\mathcal G$ such that both compositions $\mathcal{FG}$ and $\mathcal{GF}$ are naturally isomorphic to the trivial functor of the appropriate category.

A partial order $(S, \leq)$ generates a category with objects given by elements of $S$ and a single unique morphism $s\to t$ whenever $s\leq t$ in $S$. We will write $s\leq t$ for this morphism.

Given a category $\ct D$ and a \dfn{small category} $\ct C$ -- i.e. a category where the objects form a set and not a proper class -- we define a category of $\ct C$-shaped diagrams in $\ct D$ as the category $\ct D^{\ct C}$ of functors $\ct C\to\ct D$ with morphisms given by natural transformations between functors.

While the actual topology will be irrelevant to this paper, we shall mention that \dfn{homology with coefficients in a field $\kk$} is a functor from a category of topological spaces (compactly generated Haussdorff, or simplical sets, or CW-complexes, or simplicial complexes, being some of the most popular categories to use) to the category of $\kk$-vector spaces, which we denote by $\Vect_\kk$. 

\section{Persistence modules}
\label{sec:persistence-modules}

We identify the two characterizations of persistent homology with two categories of order modules.

A \dfn{real persistence module} is a diagram $(\RR, \leq)\to\Vect_\kk$. These usually appear as the homology functor images of $\RR$-parametrized diagrams of topological spaces, such as those acquired as sublevel sets of some real-valued tame function on a topological space.

A \dfn{natural persistence module} is a diagram $(\NN, \leq)\to\Vect_\kk$. These usually appear as the homology functor images of countable or finite filtrations of a topological space, such as those acquired by the sublevel set filtration of some real-valued tame function on a topological space.

We adapt the definition provided by \textcite{chazal_proximity_2009,chazal_structure_2012} of weak and strong $\varepsilon$-interleaving of persistence modules. Suppose $(\mathcal P, \leq, +)$ is a totally ordered monoid. 

Let $\ct D$ be some arbitrary category, $\mathcal P$ a totally ordered monoid, $\varepsilon\geq0$ and $M, N\in\ct D^{\mathcal P}$. We shall refer to $M, N$ as modules and write $M(p), N(p)$ for the corresponding evaluations and $M(p\leq q)$, $N(p\leq q)$ for the translation maps within $M, N$ induced by the morphism $p\leq q$ in $\mathcal P$. Recall that morphisms in $\ct D^{\mathcal P}$ are natural transformations between $M, N$ seen as functors. We shall refer to these as morphisms, however, to disambiguate layers of functors in the following. Each such category $\ct D^{\mathcal P}$ has a family of endofunctors $T_p$ defined for $p\in\mathcal P$ by $T_pM(q) = M(p+q)$. 

We define a \dfn{strong $\varepsilon$-interleaving} of $M$ and $N$ to be a pair of morphisms $f:M\to T_\varepsilon N$ and $g:N\to T_\varepsilon M$ such that\footnote{This is a slightly abuse of notation; a stringent defintion defines $T_\varepsilon$ as acting on $\mathcal P$ and describes $f, g$ as natural transformations. We choose this exposition to emphasize the similarities to how isomorphisms behave.} $fg(p)=N(p\leq p+2\varepsilon)$ and $gf(q)=M(q\leq q+2\varepsilon)$ for all $p, q\in\mathcal P$. If a strong $\varepsilon$-interleaving exists, we say that $M, N$ are strongly $\varepsilon$-interleaved.

We define a \dfn{weak $\varepsilon$-interleaving} of $M$ and $N$ to be a pair of morphisms $f:M\to T_\varepsilon N$ and $g:N\to T_\varepsilon M$ such that for some $x_0$ (dependent on the choice of morphisms) and all $k\in\NN$, $fg=N(x_0+k\varepsilon\leq x_0+(k+2)\varepsilon)$ and $gf=M(x_0+k\varepsilon\leq x_0+(k+2)\varepsilon)$. 
If a weak $\varepsilon$-interleaving exists, we say that $M, N$ are weakly $\varepsilon$-interleaved.

These $\varepsilon$-interleavings are fundamental building blocks in the various proofs of stability that abound~\cite{chazal_proximity_2009,cohen-steiner_stability_2007,chazal_stability_2008,chazal_structure_2012}. 

\section{Equivalence of categories}
\label{sec:equiv-categ}

First, we present an argument that the graded $\kk[t]$-module model of persistence modules is equivalent (as a category) to the category of natural persistence modules.
We shall then proceed to introduce a weakened equivalence notion for categories, and use this to relate natural and real persistence modules to each other.

\begin{proposition}
  The category of graded $\kk[t]$-modules is equivalent to the category of natural persistence modules.
\end{proposition}
\begin{proof}
  We shall construct an explicit equivalence between these two categories. Given a natural persistence module $M:\mathbb N\to\Vect_\kk$, we describe a graded $\kk[t]$-module $\mathcal FM$ by describing each homogenous submodule $\mathcal FM_n=M(n)$. The $\kk$-linear structure is inherited from $M(n)$, and the action of $t\cdot$ is set to $M(n\leq n+1)$. 

  Given a map $f: M\to N$ of natural persistence modules, we transfer its action pointwise to a map $\mathcal Ff:\mathcal FM\to\mathcal FN$. 

  In the other direction, we start with a graded $\kk[t]$-module $N$. We may construct a natural persistence module by setting $(\mathcal GN)(n) = N_n$. We define $\mathcal G(n\leq m)=t^{m-n}\cdot$. Again, we map morphisms by acting pointwise.

  The composition $\mathcal F\mathcal G N$ has as its degree $n$ homogenous component the vector space $(\mathcal GN)(n) = N_n$, so $\mathcal F\mathcal G N = N$ as graded vector spaces. Furthermore, the action of $t\cdot$ is given by the action of $n\leq n+1$ which is in turn defined as $t\cdot$. Hence, the modules are identical.

  Similarily, $\mathcal G\mathcal F M$ is identical to $M$ since the definitions are all performed by delegation to the previous algebraic structure.
\end{proof}

Now, since the notion of equality we want to use in order to be able to transfer stability results to the algebraic setting is that of $\varepsilon$-interleaving, we will need to modify our expectation on equivalence of categories. Hence, we shall define an \dfn{$\varepsilon$-$\delta$-interleaved equivalence} of order monoid module categories $\ct C$ and $\ct D$ to be a pair of functors $\mathcal F:\ct C\fromto\ct D :\mathcal G$ such that $\mathcal{FG}$ is naturally $\varepsilon$-interleaved with $\Id_{\ct D}$ and $\mathcal{GF}$ is naturally $\delta$-interleaved with $\Id_{\ct C}$. In other words, for any object $M\in\ct C$ we require the existence of a $\delta$-interleaving between $M$ and $\mathcal{GF}M$; for any $N\in\ct D$ we expect $N$ to be $\varepsilon$-interleaved with $\mathcal{FG}N$.

The category of real persistence modules has a subcategory of \dfn{lower stable real persistence modules}, i.e. persistence modules $M$ for which there exists some real number $x_0$ such that for all $x<y<x_0$, $M(x\leq y)$ is an isomorphism. For any given $x_0$, there is a subcategory of real persistence modules that are stable for that particular choice of $x_0$.

\begin{proposition}
  The category of lower stable real persistence modules with a given threshold value $x_0$ is equivalent to the category of lower stable real persistence modules with the threshold value of $0$.
\end{proposition}
\begin{proof}
  The functor $T_{x_0}$ has an inverse given by $\mathcal T_{-x_0}$. Hence any global translation of parameter values produces an equivalence of subcategories.
\end{proof}

\begin{proposition}\label{prop:strong-interleaving-m-tm}
  A module $M$ is strongly $\varepsilon$-interleaved with $T_\varepsilon M$.
\end{proposition}
\begin{proof}
  An interleaving is given by the family of maps $M(x)\xrightarrow{M(x\leq x+2\varepsilon)} T_\varepsilon M(x+\varepsilon)$ in the one direction, and the family of maps $T_\varepsilon M(x)\xrightarrow{M(x\leq x)}M(x+\varepsilon)$ in the other direction. 
\end{proof}

We shall later on be interested in a type of \emph{pixelation} where out of a real persistence module $M$, we create a new real persistence module that is constant over intervals $[x_0+k\varepsilon, x_0+(k+1)\varepsilon)$. Write $P_{x_0,\varepsilon}$ for the endofunctor such that $P_{x_0,\varepsilon}M(x) = M\left(\left\lfloor\frac{x-x_0}{\varepsilon}\right\rfloor\cdot\varepsilon\right)$. 
\begin{proposition}\label{prop:weak-interleaving-m-tm}
  A module $M$, stable below $x_0$, is weakly $\varepsilon$-interleaved with $P_{x_0,\varepsilon} M$.
\end{proposition}
\begin{proof}
  For a weak $\varepsilon$-interleaving of lower stable modules, it is enough to demonstrate the existence of interleaving maps for points out of some family $x_0+\mathbb N\cdot\varepsilon$. 
  
  An interleaving is given by the families of maps 
  \begin{align*}
    M(x_0+k\varepsilon)&\xrightarrow{M(x_0+k\varepsilon\leq x_0+(k+2)\varepsilon)} T_\varepsilon M(x_0+(k+1)\varepsilon) \\
    T_\varepsilon
    M(x_0+k\varepsilon)&\xrightarrow{M(x_0+k\varepsilon\leq
      x_0+k\varepsilon)}M(x_0+k\varepsilon)
  \end{align*}
\end{proof}

\begin{proposition}
  A weak $\varepsilon$-interleaving gives rise to a strong $2\varepsilon$-interleaving.
\end{proposition}
\begin{proof}
  Suppose $M$ and $N$ are weakly $\varepsilon$-interleaved with basepoint $x_0$, interval size $\varepsilon$, and functions $f:M\to T_\varepsilon N$, $g:N\to T_\varepsilon M$.

  A strong $2\varepsilon$-interleaving is given by 

\begin{align*}
  M(x) &\to  M\left(\left\lceil\frac{x-x_0}{\varepsilon}\right\rceil\cdot\varepsilon\right)\xrightarrow{f}N\left(\left\lceil\frac{x-x_0}{\varepsilon}\right\rceil\cdot\varepsilon+\varepsilon\right)\to
  N(x+2\varepsilon) \\
  N(x) &\to  N\left(\left\lceil\frac{x-x_0}{\varepsilon}\right\rceil\cdot\varepsilon\right)\xrightarrow{g}M\left(\left\lceil\frac{x-x_0}{\varepsilon}\right\rceil\cdot\varepsilon+\varepsilon\right)\to
  M(x+2\varepsilon) \\
\end{align*}

\end{proof}

The approach we are giving in the following theorem to an equivalence is very similar to the $\varepsilon$-periodic pixelization described by \textcite{chazal_proximity_2009}. One fundamental difference is that we here define a new real persistence module which is constant between pixelization indices, allowing us to use Proposition~\ref{prop:weak-interleaving-m-tm} to prove weak interleaving.

\begin{theorem}
  For any $\varepsilon>0$, the category of $0$-based lower stable real persistence modules is weakly $2\varepsilon$-$2$-interleaved equivalent with the category of natural persistence modules.
\end{theorem}
\begin{proof}
  Fix some $\varepsilon>0$. We shall define a pair of functors as follows: 

  For a real persistence module $M$, we define a natural persistence module $\mathcal FM$ by $\mathcal FM(n) = M((n+1)\cdot\varepsilon)$. The map $\mathcal FM(n\leq m)$ is defined to be $M((n+1)\cdot\varepsilon\leq(m+1)\cdot\varepsilon)$. 

  For a natural persistence module $N$, we define a real persistence module $\mathcal GN$ by $\mathcal GN(x) = N\left(\left\lfloor\frac x \varepsilon \right\rfloor+1\right)$. 

  \textbf{Claim} $\mathcal{GF}M$ is weakly $2\varepsilon$-interleaved with $M$. \\
  Indeed, $\mathcal{GF}M(x) = M\left(\left(\left\lfloor\frac x\varepsilon\right\rfloor+2\right)\cdot\varepsilon\right)$. 
So in particular, $\mathcal{GF}M(n\varepsilon)=M((n+2)\varepsilon)$.

We note that this situation is almost identical to the situation in Proposition~\ref{prop:strong-interleaving-m-tm}; but due to the pixelation in effect from the floor computation, we only recover a weak interleaving, using Proposition~\ref{prop:weak-interleaving-m-tm}.

  \textbf{Claim} $\mathcal{FG}N$ is strongly $2$-interleaved with $N$. Indeed, $\mathcal{FG}N(n) = N\left(\left(\left\lfloor\frac{(n+1)\cdot\varepsilon}{\varepsilon}\right\rfloor +1\right)\right)=N(n+2)$.

  The result follows.
\end{proof}

\begin{corollary}
  For any $\varepsilon>0$, the category of $0$-based lower stable real persistence modules is strongly $4\varepsilon$-$4$-interleaved equivalent with the category of natural persistence modules.  
\end{corollary}

\section{Software considerations}
\label{sec:softw-cons}

It is useful at this point to consider the software landscape. Most algorithms for persistent homology expect a \emph{simplex stream} -- a total refinement of the inclusion partial order on simplices in a simplicial complex, compatible as an order with the order induced by the filtration. For these, the algorithms then step one step for each new simplex added, and at the end of the computation translate the natural numbered indices back to the original \emph{filtration values}. This is how all persistent homology software we know operates~\cite{morozov_dionysus_2011,sexton_jplex_2008,tausz_javaplex:_2012}.

In this framework, this operation mode corresponds to taking categorical skeleta on the algebraic side; all runs of isomorphisms get pulled together to a single representative, and only the transitions between these runs are given space in the computation. After the algorithm finishes, these chunks are unpacked back into longer runs of modules, and translated back along the equivalence over to the real persistence module side of the operation at the same time.

\section{Concluding remarks}
\label{sec:concluding-remarks}

We have demonstrated that up to $\varepsilon$-interleavings, the two major categories used for classical persistent homology are equivalent. This means, in particular, that results on either side transfer to the other side stringently; a decomposition of a natural persistence module into interval modules, say, produces a corresponding decomposition of a real persistence module. On the other hand, a stability result for a real persistence module applies to its corresponding natural persistence module. For the equivalence, we have to settle on a scale, and on a lower bound, but the results hold regardless of the particular choice made.

The lower bound is only really chosen to enable us to build an interleaved equivalence to diagrams of the shape $(\NN, \leq)$. If we were to remove the lower bound requirement, these arguments would instead demonstrate an interleaved equivalence to diagrams of the shape $(\mathbb Z, \leq)$. These are, however, of much less use in current persistent homology.

It is also worth pointing out that the equivalences demonstrated do not in any way require tameness results to hold for the modules or their origins. Actual stability results or algebraic structure theorems will require additional prerequisites, but the interleaved equivalences hold regardless of the structures of the constituent modules.

\section{Acknowledgements}
\label{sec:acknowledgements}

Conversations with both Wojciech Chacholski and Primoz Skraba have refined and clarified important steps of the arguments in this paper.

The research was funded entirely by TOPOSYS (FP7-ICT-318493-STREP).

\newpage
\printbibliography

\end{document}